\documentclass[12pt]{article}
\usepackage{amsmath}
\usepackage{amssymb}
\usepackage{amsfonts}
\usepackage{comment}
\usepackage{color}
\usepackage{mathtools}
\usepackage{tikz}
\DeclarePairedDelimiter{\ceil}{\lceil}{\rceil}
\DeclarePairedDelimiter{\floor}{\lfloor}{\rfloor}
\usepackage{amsthm}

\newtheorem*{notation*}{Notation}
\newtheorem{definition}{Definition}
\newtheorem{theorem}{Theorem}[section]
\newtheorem{proposition}[theorem]{Proposition}
\newtheorem{lemma}[theorem]{Lemma}
\newtheorem{corollary}[theorem]{Corollary}
\newtheorem{algorithm}{Algorithm}
\newtheorem*{remark}{Remark}
\numberwithin{equation}{section}
\DeclareMathOperator{\sat}{sat}
\DeclareMathOperator{\cov}{cov}

\def \L {\mathcal{L}}
\def \M {\mathcal{M}}
\def \S {\mathcal{S}}
\def \A {\mathcal{A}}

\tikzstyle{cir} = [draw, circle, minimum height= 20 mm]

\title{Extremal graphs with local covering conditions}
\author{Debsoumya Chakraborti\thanks{Department of Mathematical Sciences, Carnegie Mellon University. Email: {\tt dchakrab@cmu.edu}. Research supported in part by National Science Foundation CAREER Grant DMS-1455125.} \ and Po-Shen Loh\thanks{Department of Mathematical Sciences, Carnegie Mellon University. Email: {\tt ploh@cmu.edu}. Research supported in part by National Science Foundation CAREER Grant DMS-1455125.}}
\begin{document}
\maketitle
\begin{abstract}

We systematically study a natural problem in extremal graph theory, to
minimize the number of edges in a graph with a fixed number of vertices,
subject to a certain local condition: each vertex must be in a copy of a
fixed graph $H$. We completely solve this problem when $H$ is a clique, as
well as more generally when $H$ is any regular graph with degree at least
about half its number of vertices. We also characterize the extremal graphs 
when $H$ is an Erd\H{o}s-R\'enyi random graph. The extremal
structures turn out to have the similar form as the conjectured extremal
structures for a well-studied but elusive problem of similar flavor with
local constraints: to maximize the number of copies of a fixed clique in
graphs in which all degrees have a fixed upper bound.

\end{abstract}

\section{Introduction}
\subsection{Problem and motivation}

Extremal graph theory considers problems of maximizing or minimizing certain 
graph parameters in graph classes of interest. The most classical
example is the Tur\'an problem of maximizing the number of edges in a graph
with a given number of vertices, subject to the condition of being $H$-free
(having no subgraph isomorphic to a fixed graph $H$). Tur\'an solved the
problem completely when $H$ is a complete graph in \cite{T}, and for
general $H$ with chromatic number at least 3, the Erd\H os-Stone-Simonovits
Theorem \cite{ES} provides an asymptotic answer. When $H$ has chromatic
number 2, the problem is more intricate, with many longstanding open
questions remaining (see, e.g., the surveys \cite{FS} and \cite{S}).

A variety of questions of similar flavor have been the focus of significant
research attention. When one generalizes the problem to maximizing the
number of cliques of order $t \ge 3$ in an $n$-vertex $H$-free graph,
even the basic question where $H = K_{1, D}$ is a star (translating into a
maximum degree condition of $D-1$) is not completely understood.  Indeed,
if $n = qD+r$ ($q \in \mathbb{N}$ and $1 \le r \le D$), the maximum number
of cliques $K_t$ is conjectured to be obtained by the disjoint union of $q$
many $K_{D}$ and one $K_r$. Although for $q = 1$ this conjecture is
resolved in \cite{GLS}, the general problem is still wide open despite substantial 
effort (see, eg., \cite{CR}). It appears to be difficult to prove extremality of this 
type of structure, with many disjoint copies of the same clique, and one 
residual graph which depends on the residue class of $n$.

An equivalent form of the above conjecture was asked by Engbers and Galvin in \cite{EG}.
In this equivalent form, we seek the maximum number of independent sets 
of order $t$ in an $n$-vertex graph with minimum degree at least $\delta$.
It is easy to see the equivalence between this statement and the aforementioned 
conjecture by considering the complement of a graph. This is now the same as
asking for the maximum number of independent sets of order $t$ in an $n$-vertex 
graph where every vertex $v$ is in a copy of $K_{1,\delta}$ with $v$ being the 
root of $K_{1,\delta}$. It is natural to consider the same problem with $K_{1,\delta}$ 
replaced by a general graph $H$. In this paper, we will consider a 
weaker condition where we do not require that each vertex $v$ 
in $G$ is a specific vertex in a copy of $H$. Note that this does not change 
the problem when $H$ is a complete graph. It turns out that even the case 
$t=2$ for this problem is non-trivial and has some surprising results, which
will be the main topic of this article. Note that maximizing the number of independent 
sets of order $2$ is equivalent to minimizing the number of edges in a graph.  

On the topic of $H$-free graphs, the question of minimizing the number of
edges becomes interesting when one adds the further property that adding
one more edge to $G$ will create a copy of $H$. This motivates the definition of $\sat(n,H)$, 
which is the minimum number of edges in an $n$-vertex $H$-free graph $G$, with that property. 
Erd\H{o}s, Hajnal, and Moon \cite{EHM} answered this question when $H$ is a complete
graph, and it is known that if $H$ is a graph on $p$ vertices then
$\sat(n,H) \le \sat(n,K_p) = O(n)$, but it is not even known if $\lim_{n
\rightarrow \infty} \frac{\sat(n,H)}{n}$ exists for every graph $H$ (Tuza's
conjecture \cite{TT}). This class of problems is extensively surveyed in
\cite{FFS}.

In this paper, we consider a natural minimization problem which has similar
flavor. We study the problem of minimizing the number of edges (or more
generally, copies of $K_t$) in a graph $G$ with $n = qD+r$ vertices, subject to the
condition that it is \textbf{$\boldsymbol{H}$-covered} by some fixed graph
$H$, i.e., each vertex in $G$ is in a copy of $H$. We completely solve the problem
of determining the unique $K_{D}$-covered graph with any given number of
vertices that minimizes the number of cliques of order $t$. It
turns out that the edge-minimal graph has almost the same structure as in
the conjecture in \cite{GLS}, mentioned at the beginning of this paper. 
This time, the residual graph is not $K_r$, but instead is a union of two copies of 
$K_{D}$ overlapping in $D-r$ vertices, as stated formally below.

\subsection{Main results}
\begin{proposition}\label{pr}
For any positive integers $q,n,t$ with $2 \le t \le n$, and any integer $N = qn + r$ with $0 \le r < n$, the graph consisting of the union of 2 copies of $K_n$ sharing $n-r$ vertices, together with the disjoint union of $q-1$ many $K_n$, has the least number of copies of $K_t$ among all $K_n$-covered graphs on $N$ vertices. Moreover, this is the unique such graph if $n-r \ge t$ or $n-r = 1$.
\end{proposition}

\begin{remark}
If $n-r < t$, then any graph formed by union (not disjoint union) of $q+1$ many $K_n$'s that spans exactly $qn + r$ vertices will be an extremal graph. 
\end{remark}

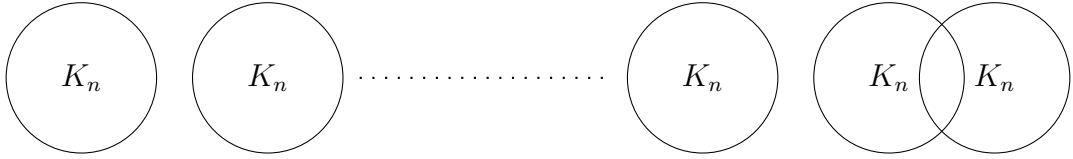
\begin{figure}[h]
\centering
\begin{tikzpicture}
\node[cir](cov){$K_n$};
\node[cir, xshift = 6em](cov1){$K_n$};
\draw [loosely dotted, style = thick] (3.7,0) -- (7,0);
\node[cir, xshift = 20em](cov2){$K_n$};
\node[cir, xshift = 26em](cov3){$K_n$};
\node[cir, xshift = 29.4 em](cov4){$K_n$};
\end{tikzpicture}
\caption{$K_t$-minimal $K_n$-covered graph}\label{diag}
\end{figure}

The first natural question to ask is the following. Are there other graphs $H$ that give rise to the same structure, with many disjoint copies of $H$, and a bounded residual graph? We first show a simple example of a graph which does not. Consider the graph $H$ consisting of the union of $K_n$ and $K_2$ overlapping in one vertex, i.e., $K_n$ with one pendant vertex. A moment's reflection will convince the reader that an $H$-covered graph $G$ is edge-minimal if and only if $G$ consists of a copy of $K_n$, together with $N-n$ more edges which connect the rest of the $N-n$ vertices to the $K_n$ with a single edge each, i.e., $K_n$ with $N-n$ pendant vertices. One can also prove this by using Theorem \ref{th1} which will be stated towards the end of this introduction. In order to generalize Proposition \ref{pr} to other graphs, we need a couple of definitions to describe the residual ``last piece'' (the analogue of the two overlapping copies of $K_n$). 

\begin{definition}\label{defnA}
Suppose $H$ and $T$ are graphs with $A \subseteq V(H)$ where $V(H)$ denotes the vertex set of $H$. Let $a_T^H(A)$ denote the number of subgraphs isomorphic to $T$ in $H$ which have at least one vertex in $A$. Also, let $a_T^H(k)$ denote the minimum value of $a_T^H(A)$ over all k-vertex subsets $A \subseteq V(H)$. For convenience, we define $a_T^H(0) = 0$. When $H$ is understood from the context, we will drop $H$ from the notation. Also, when $T = K_2$, we will drop $T$ from the notation. 
\end{definition}

When generalizing the ``last piece'' of Proposition \ref{pr}, it turns out that there may be several options. Intuitively the last piece is just 2 copies of $H$ overlapping in the most efficient way to minimize the total number of edges. We write $\sqcup$ to denote disjoint union throughout this paper.

\begin{definition}\label{intersect}
Suppose $T$ is a complete graph, $H$ is a graph with $n$ vertices and $N = n + r$ with $0 \le r < n$. Define $\L_{T,N}^H$ to be the collection of graphs $G$ on the vertex set $V \sqcup V_1 \sqcup V_2$ with $|V| = n-r$ and $|V_1| = |V_2| = r$ satisfying the following. There are no edges between $V_1$ and $V_2$. The set $V$ induces a subgraph of $H$ with $n-r$ vertices which has the maximum possible number of subgraphs isomorphic to $T$, and $V_1 \sqcup V$ and $V_2 \sqcup V$ both induce $H$. When $T = K_2$, we will drop $T$ from the notation.
\end{definition}

Note that in Definition \ref{intersect}, the number of subgraphs isomorphic to $T$ in $V$ is $a_T^H(n) - a_T^H(r)$. Hence the number of subgraphs isomorphic to $T$ in any graph in $\L_{T,N}^H$ is $a_T^H(n) +  a_T^H(r)$. Observe that when $H = K_n$, the unique graph in $\L_N^H$ is just the last piece of the graph in Figure \ref{diag}. We now define a notion that intuitively represents any graph which gives rise to similar structures to what the complete graph gave for the edge minimization problem of $H$-covered graphs.

\begin{definition}\label{defn1}
Suppose $H$ is a graph with $n$ vertices and $r$ is an integer with $0 \le r < n$. We call a graph $H$ \textbf{ideal in remainder class} $\boldsymbol{r}$ if for all $N = qn + r$ with $q \in \mathbb{Z}^+$, a graph $G$ is edge-minimal among $N$-vertex $H$-covered graphs if and only if $G$ consists of a graph in $\L_{n+r}^H$, together with the disjoint union of $q-1$ copies of $H$. For any subset $S \subseteq \{0,1, \dots, n-1\}$, we call a graph $H$ \textbf{ideal in remainder class} $\boldsymbol{S}$ if it is ideal in remainder class $r$ for all $r \in S$. We call a graph $H$ \textbf{ideal} if it is ideal in remainder class $\{0, 1, \dots, n-1\}$. 
\end{definition}

Proposition \ref{pr} states that $K_n$ is an ideal graph. There are, however, graphs that are not ideal, as we saw before. As a more general example, we can show that for any graph $H$ whose two lowest-degree vertices have degrees $d_1$ and $d_2$ with $d_1 \ge d_2 + 2$, this type of structure does not minimize the number of edges. For such a graph $H$, any graph in $\L_{n+2}^H$ has at least $e(H) + d_1 + d_2 - 1$ many edges, where $e(H)$ denotes the number of edges in $H$. On the other hand, consider the graph $M$ created by the union of a copy of $H$ and $2$ more vertices which mimic the minimum degree in $H$, i.e., these 2 vertices are adjacent to a vertex $v$ of $H$ if and only if $v$ is adjacent to the vertex with minimum degree $d_1$. It can easily be seen that $M$ is $H$-covered and has $e(H) + 2d_1$ many edges, which is less than the number of edges in graphs in $\L_{n+2}^H$. Hence for $N = qn + 2$ with $q \in \mathbb{N}$, we will have strictly less edges if we use $M$ as the last piece instead of any graph in $\L_{n+2}^H$, so $H$ is not ideal in remainder class $\{2\}$. Also, for a similar reason, if the two lowest-degree vertices have degrees $d_2 = d_1 + 1$, then the structure in Definition \ref{defn1} may not always be the unique graph for our problem. Hence it is quite natural to study regular graphs, and we have the following theorem.

\begin{theorem}\label{th2}
Every $n$-vertex $d$-regular graph with $d \geq \frac{n-1}{2}$ is ideal. This is tight in a sense that there exist non-ideal $d$-regular $n$-vertex graphs when $d$ is the largest even integer less than $\frac{n-1}{2}$. 
\end{theorem}

The structures for the edge minimization problem with respect to non-regular graphs $H$ can be very different in nature, and it seems that there is no systematic way to generalize and classify them. But surprisingly, we can characterize the problem for random graphs $H$ quite well. We consider the Erd\H os-R\'enyi random graph model $G_{n,p}$ for constant $p$, where each potential edge among $n$ vertices independently appears with probability $p$. Throughout this paper, we say that event $A_n$ occurs \textbf{with high probability} if the probability of that event ($\mathbb{P}[A_n]$ in notation) tends to one as $n \rightarrow \infty$. Although $G_{n,p}$ is not ideal (in particular $G_{n,p}$ is not ideal in remainder class $\{2\}$ because of the well known fact that with high probability the two lowest degrees differ by much more than $2$), it turns out that with high probability $G_{n,p}$ is ideal in large remainder classes, and we can get another interesting kind of structure for small remainder classes. 
Motivated by the structure that appeared when the difference between the two lowest degrees of $H$ was 2, we define a couple of notions to help study the the problem for small remainder classes when $H$ is a random graph. Throughout the paper, we use the standard notation $[n]$ to represent the set $\{1, 2, \cdots, n\}$.

\begin{definition}\label{definition}
Suppose $H$ is a graph with $n$ vertices and $n \leq N < 2n$. Define $\M_N^H$ to be the collection of graphs $G$ satisfying the following: $G$ is a graph on the vertex set $[N]$, where for all $n \le i \le N$, the vertex subset $[n-1] \cup \{i\}$ induces a copy of $H$ and the number of edges in $G$ is the minimum possible. 
\end{definition}

\begin{figure}[h]
\centering 
\begin{tikzpicture}
\draw (0,0) circle (2);
\draw [thick] (3,1) -- (1,0.3);
\draw [thick] (3,1) -- (1,-0.3);
\draw [thick] (3,1) -- (1,0.9);
\draw [thick] (3,1) -- (1,-0.9);
\draw [thick] (3,0) -- (1,0.3);
\draw [thick] (3,0) -- (1,-0.3);
\draw [thick] (3,0) -- (1,0.9);
\draw [thick] (3,0) -- (1,-0.9);
\draw [thick] (3,-1) -- (1,0.3);
\draw [thick] (3,-1) -- (1,-0.3);
\draw [thick] (3,-1) -- (1,0.9);
\draw [thick] (3,-1) -- (1,-0.9);
\node at (-0.5,0) {$H \setminus \{n\}$};
\node [right] at (3,1) {$n$};
\node [right] at (3,0) {$n+1$};
\node [right] at (3,-1) {$n+2$};
\node [left] at (1,0.9) {$u$};
\node [left] at (1,0.3) {$v$};
\node [left] at (1,-0.3) {$w$};
\node [left] at (1,-0.9) {$z$};
\end{tikzpicture}
\caption{Example of a graph in $\M_{n+2}^H$}\label{diagram}
\end{figure}
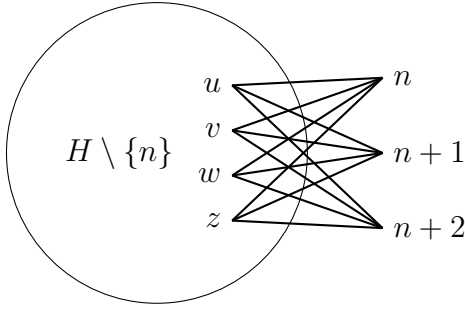

In Definition \ref{definition} of $\M_N^H$, clearly there is no edge in $G$ between $i,j$ for all $n \le i < j \le N$, and for all $n \le i \le N$, the degree of $i$ in $G$ is the minimum degree in $H$. Note that $\M_n^H = \{H\}$. In Figure \ref{diagram}, $H$ is a graph on vertex set $[n]$ ($n \ge 3$) with $n$ being a vertex with minimum degree. We obtained the graph in the figure by adding 2 more vertices $n+1, n+2$ and mimicking the vertex $n$, i.e., for each $u \in [n-1]$, adding edges $(n+1,u)$ and $(n+2,u)$ if and only if there is an edge $(n,u)$. Note that if there is more than one vertex with minimum degree, the extra vertices (e.g., $n+1, n+2$ in Figure \ref{diagram}) can choose to mimic any one of them. On the other hand, if $H$ has a unique vertex with minimum degree, then $\M_N^H$ has only one graph for each $n \le N < 2n$ and the graph will look like the one in Figure \ref{diagram}.  

\begin{definition}\label{defn2}
Suppose $H$ is a graph with $n$ vertices and $r$ be an integer with $0 \le r < n$. We call a graph $H$ \textbf{elementary in remainder class} $\boldsymbol{r}$ if for all $N = qn + r$ with $q \in \mathbb{Z^+}$, a graph $G$ is edge-minimal among $N$-vertex $H$-covered graphs if and only if $G$ is the disjoint union of q many graphs, each from $\M_{b_i}^H$ for some $n \le b_i < 2n$, such that $\sum_{i=1}^q b_i = N$. For any subset $S \subseteq \{0,1, \dots, n-1\}$, we call a graph $H$ \textbf{elementary in remainder class} $\boldsymbol{S}$ if it is elementary in remainder class $r$ for all $r \in S$. We call a graph $H$ \textbf{elementary} if it is elementary in remainder class $\{0, 1, \dots, n-1\}$.
\end{definition}

Note that for remainder class $\{0,1\}$, $H$ is elementary if and only if $H$ is ideal, because $\M_{n}^H = \L_{n}^H$ and $\M_{n+1}^H = \L_{n+1}^H$, where $n$ is the number of vertices of $H$.

\begin{theorem}\label{th3}
Let $0 < p < 1$ be a constant. With high probability, $G_{n,p}$ is elementary in remainder class $\{0,1,...,\gamma\}$ and is ideal in remainder class $\{\beta, \beta + 1,..., n-1\}$ for some $\gamma, \beta$ with $\gamma = \Theta(\sqrt{n \log n})$ and $\beta = \gamma + 1$ or $\beta = \gamma + 2$.
\end{theorem}

\begin{remark}
Since $\gamma < \beta \le \gamma + 2$, with high probability there exists at most one $r = \Theta(\sqrt{n \log n})$ such that $G_{n,p}$ is neither elementary nor ideal in remainder class $r$. Furthermore it turns out that with high probability, if such $r$ exists, then for the edge minimization problem of $G_{n,p}$-covered graphs, both the graph structures in Definition \ref{defn1} and Definition \ref{defn2} work, and those are the only graphs as well. We will also give the condition on when such $r$ exists, but we postpone the detailed analysis until the proof of Theorem \ref{th3} in Section \ref{theorem3}.
\end{remark}

Consider the more general problem of minimizing the number of cliques of an arbitrary fixed order instead of edges ($K_2$) among $H$-covered graphs with a given number of vertices. We will formulate this general problem in terms of an integer program which will be useful to prove Theorem \ref{th2} and Theorem \ref{th3}.

\begin{theorem}\label{th1}
Fix positive integers $N$ and $t$, and a graph $H$ with $n$ vertices. The minimum number of subgraphs isomorphic to $K = K_t$ among all $H$-covered graphs with $N$ vertices, denoted by $\cov_{K_t}(N,H)$, is equal to the solution to the following integer program:

\begin{align} \label{IP}
&\text{minimize} \displaystyle \sum_{k \in [n]} a_K^H(k) x_k, \nonumber\\ 
&\text{subject to} 
\begin{cases}
\displaystyle \sum_{k \in [n]} k x_k = N \\
x_k \in \mathbb{Z} &\text{for all	} k \in [n] \\
x_k \ge 0 &\text{for all	} k \in [n-1] \\
x_n \ge 1. 
\end{cases}
\end{align}
\end{theorem}

Now we state a corollary of Theorem \ref{th1}, which can be related to the corresponding graph saturation problem (Tuza's Conjecture \cite{TT}), which as mentioned in the beginning, has remained open for a long time.   

\begin{corollary}\label{sat}
For every graph $H$ and every positive integer $t$, $\lim_{N \rightarrow \infty} \frac{\cov_{K_t}(N,H)}{N}$ exists.
\end{corollary}

We will use the following standard graph theoretic notation throughout this paper. For a graph $G$ and $v \in V(G)$, $N_G(v)$ denotes the set of neighbors of $v$ in $G$ and $d_G(v) = |N_G(v)|$, i.e., $d_G(v)$ is the degree of $v$. We write $\delta(G)$ to denote the minimum degree of $G$. This paper is organized as follows. In Section \ref{sec2}, we prove Proposition \ref{pr} and Theorem \ref{th1}. In Section \ref{sec3}, we prove Theorem \ref{th2} and construct graphs to show its tightness. In Section \ref{sec4}, we discuss some facts and lemmas about random graphs as a preparation to prove Theorem \ref{th3}. Finally, we prove it in Section \ref{sec5} and end with some concluding remarks in Section \ref{sec6}.

\section{Minimizing the number of cliques of a fixed order}\label{sec2}
We start with a well-known result.
\begin{lemma}\label{conv}
Suppose $n$, $k$, $r$, and $m$ are non-negative integers with $r < n$ and $m > k$. Over all choices of values $x_i \in \{0,1,2,...,n\}$ for each $i \in [m]$ such that $\sum_{i=1}^m x_i = kn + r$, the expression $\sum_{i=1}^t \binom{x_i}{2}$ achieves its maximum if and only if $x_i = n$ for $k$ many values of $i \in [m]$, $x_i = r$ for one $i$, and $x_i = 0$ for the rest of the values of $i$.
\end{lemma}
\begin{proof}
The function $f : \mathbb{Z} \rightarrow \mathbb{R}$ defined by $f(x) = \frac{x(x-1)}{2}$ is strictly convex, so $f(a) + f(b) < f(a-1) + f(b+1)$ whenever $a < b$. Hence, if the values of $x_i$ are not as in Lemma \ref{conv}, we can increase the value of $\sum_{i=1}^m \binom{x_i}{2}$ by changing 2 of $x_i$ while keeping $\sum_{i=1}^m x_i$ constant, which is a contradiction.
\end{proof}

\begin{proof}[\textbf{Warm up for Proposition \ref{pr} when $\boldsymbol{t=2}$}]
Fix $q$, $n$ and $r$ with $r < n$, and let $N = qn + r$. Let $F$ denote the graph which consists of the union of 2 copies of $K_n$ overlapping in $n-r$ vertices, together with the disjoint union of $q-1$ many $K_n$. First note that the number of edges in the graph $F$ is $(q+1) \binom{n}{2} - \binom{n-r}{2}$. Now our goal is to show that any $N$-vertex $K_n$-covered graph has at least $(q+1) \binom{n}{2} - \binom{n-r}{2}$ many edges.

Consider a $K_n$-covered graph $G$ with $N$ vertices. Since $v$ is in a copy of $K_n$ for all $v \in V(G)$, choose a set $S(v)$ of $n$ vertices of $G$ such that $v \in S(v)$ and $S(v)$ induces $K_n$. Now in order to lower-bound the number of edges in $G$, do the following:

\begin{algorithm} \label{alg}\ \smallskip 											
\begin{enumerate}
\item Set $V_1 = V(G)$, $E_1 = E(G)$.
\item In the $i$-th step, pick an arbitrary vertex $v_i \in V_i$. Consider the subgraph $A(v_i) \cong K_n$ induced by $S(v_i)$, Now set $V_{i+1} = V_i \setminus S(v_i)$ and $E_{i+1} = E_i \setminus E(A(v_i))$. Repeat the procedure until $V_{i+1} = \emptyset$. 
\end{enumerate}
\end{algorithm}

Note that in each step $i$, $E_i \supseteq E(G) \setminus E(V(G) \setminus V_i)$. So, the number of edges in $A(v_i) \cap E_i$ is at least $\binom{n}{2} - \binom{k_i}{2}$, where $k_i = |(V(G) \setminus V_i) \cap S(v_i)|$. Let $m$ be the number of iterations, i.e. $V_{m+1} = \emptyset$. Observe that $q \le m \le N$. Clearly $N = \sum_{i=1}^m (n - k_i)$. and the number of edges in $G$ is at least $\sum_{i=1}^m \big(\binom{n}{2} - \binom{k_i}{2}\big)$. So, the only remaining thing is to determine the maximum value of $\sum_{i=1}^m \binom{k_i}{2}$ when $k_i \leq n$ for all $i \in [m]$ and $$\sum_{i = 0}^m k_i = mn - N = (m - q - 1)n + (n - r).$$ 
By using Lemma \ref{conv}, $\sum_{i=1}^m \binom{k_i}{2}$ attains its maximum if and only if $k_i = n$ for $m-q-1$ many values of $i$, $k_i = n-r$ for some other $i$, and $k_i = 0$ for the rest of the values of $i$. Hence, the number of edges in $G$ is at least $(q+1) \binom{n}{2} - \binom{n-r}{2}$, proving the required bound. Now in order to prove the uniqueness of the extremal graph $F$, let us have a closer look at Algorithm \ref{alg}. As we have already noticed that $F$ has $(q+1) \binom{n}{2} - \binom{n-r}{2}$ many edges, if we assume that $G$ has the minimum number of edges among $N$-vertex $K_n$-covered graphs, then $E_{m+1} = \emptyset$. By Lemma \ref{conv}, the values of $k_i$'s are unique up to permutation, and they are exactly same as what we have seen in the beginning of this paragraph. Hence, it is straightforward to see that $F$ is the only graph achieving the minimum number of edges among $N$-vertex $K_n$-covered graphs.
\end{proof}

\begin{proof}[\textbf{Proof of Theorem \ref{th1}}]
Let the minimum value of the integer program \eqref{IP} be $z$. We first show that there is an $N$-vertex $H$-covered graph with $z$ many copies of $K = K_t$. Let  $x_k$ be non-negative integer for all $k \in [n]$ such that $x_n \geq 1$ and $\sum_{k=1}^n k x_k = N$. We will construct an $N$-vertex $H$-covered graph with the number of copies of $K$ being $\sum_{k \in [n]} a_K(k) x_k$, where $a_K$ is a short hand for $a_K^H$.

Let $y_k = x_k$ for $k < n$ and let $y_n = x_n - 1$. Consider a graph $G$ on the vertex set $V \sqcup (\sqcup_{k=1}^n V_k)$, where $V_k = \sqcup_{i=1}^{y_k} V_{i,k}$, $V$ has $n$ vertices and each $V_{i,k}$ has $k$ vertices. There are no edges between $V_{i,k}$ and $V_{j,l}$ in $G$ when $(i,k) \neq (j,l)$. $V$ induces $H$ and $V \sqcup V_{i,k}$ induces a graph in $\L_{K,n+k}^H$. From the discussion after Definition \ref{intersect}, remember that the number of subgraphs isomorphic to $K$ in $\L_{K,n+k}^H$ is $a_K(n) + a_K(k)$, hence the number of copies of $K$ in $G$ is $\sum_{k \in [n]} a_K(k) x_k$.

At this point, to finish the proof of Theorem \ref{th1}, it will be enough to show that the number of copies of $K$ of an $N$-vertex $H$-covered graph is at least $z$. We will do a similar treatment as it was done in the proof of Proposition \ref{pr} where $H$ was a complete graph and $K$ was $K_2$. Consider an $N$-vertex $H$-covered graph $G$. So, $v$ is in a copy of $H$ for all $v \in V(G)$, choose $S(v)$ a set of $n$ vertices of $G$ such that $v \in S(v)$ and $H$ is a subgraph of the induced subgraph on $S(v)$. Now we follow the following steps to lower-bound the number of copies of $K$ in $G$, which is similar to Algorithm \ref{alg} in the proof of Proposition \ref{pr} for $t = 2$.

\begin{algorithm}\label{algo}\ \smallskip
\begin{enumerate}
\item $V_1 = V(G)$, $E_1 = E(G)$.
\item In the $i$-th step, pick an arbitrary vertex $v_i \in V_i$. Consider the subgraph $A(v_i) \supseteq H$ induced by $S(v_i)$. Now set $V_{i+1} = V_i \setminus S(v_i)$ and $E_{i+1} = E_i \setminus E(A(v_i))$. Repeat the procedure if $V_{i+1} \neq \emptyset$.
\end{enumerate}
\end{algorithm}

By using Definition \ref{defnA}, the number of subgraphs isomorphic to $K$ in $A(v_i) \cap E_i$ is at least $a_K(k_i)$, where $k_i = |V_i \cap S(v_i)|$. Note that $k_1 = n$. Let $m$ be the number of iterations, i.e. $V_{m+1} = \emptyset$. Note that $q \le m \le N$. Clearly $\sum_{i=1}^m k_i = N$ and the number of subgraphs isomorphic to $K$ in $G$ is at least $\sum_{i=1}^m a_K(k_i)$. Hence, the number of subgraphs isomorphic to $K$ in $G$ is at least the minimum value of the integer program \eqref{IP}.
\end{proof}

\begin{proof}[\textbf{Proof of Corollary \ref{sat}}]
Let $c = \min_{k \in [n]} \frac{a_{K_t}^H(k)}{k}$. From the integer program \eqref{IP}, one can observe that $$cN \le \cov_{K_t}(N,H) \le a_{K_t}^H(n) + cN,$$
which finishes the proof of Corollary \ref{sat}.
\end{proof}

\begin{proof}[\textbf{Proof of Proposition \ref{pr}}]
Note that $a_{K_t}^{K_n}(k) = \binom{n}{t} - \binom{n-k}{t}$. So, by using the fact that $\binom{x}{t} + \binom{y}{t} > \binom{z}{t} + \binom{w}{t}$ for all $x,y,z,w \in \mathbb{N}$ such that $x + y = z + w$, $x > \max(z,w)$, and $x \ge t$, we get the following 2 equations for $0 < k, l < n$.
\begin{equation}
a_{K_t}^{K_n}(k) + a_{K_t}^{K_n}(l) > a_{K_t}^{K_n}(k + l), \text{          for          } k + l \le n \text{ and } n \ge t \label{conv_type_ineq_1} 
\end{equation}
\begin{equation} 
a_{K_t}^{K_n}(k) + a_{K_t}^{K_n}(l) > a_{K_t}^{K_n}(n) + a_{K_t}^{K_n}(b), \text{          for          } n < k + l = n + b < 2n \text{ and } n-b \ge t \label{conv_type_ineq_2}
\end{equation}

Hence, if $r \neq 0$, then the integer program \eqref{IP} in Theorem \ref{th1} has one unique solution: $x_n = q$, $x_r = 1$, and $x_k = 0$ for other values of $k \in [n]$. For $r=0$, the unique solution is just $x_n = q$ and $x_k = 0$ for $k \neq n$. Now for an $N$-vertex $K_n$-covered graph $G$, by analyzing the steps of Algorithm \ref{algo} similarly to the warm up proof of Proposition \ref{pr} for $t=2$, we can conclude that $G$ minimizes the number of subgraphs isomorphic to $K_t$ among $N$-vertex $K_n$-covered graphs if and only if $G$ is the extremal graph in Proposition \ref{pr}.
\end{proof}

\section{Regular graphs}\label{sec3}
\subsection{Proof of Theorem \ref{th2}}
Throughout this section, let $H$ be a fixed graph on $n$ vertices so that we may write $a(k)$ instead of $a^H(k)$. We first warm up by proving the following interesting result. 

\begin{proposition} \label{extraadded}
For any positive integers $q,n,N = qn$ and any $n$-vertex graph $H$ with $\delta(H) \ge \frac{n-1}{2}$, $G$ is the edge-minimal $N$-vertex $H$-covered graph if and only if $G$ is the disjoint union of $q$ many copies of $H$. 
\end{proposition}

\begin{remark}
In contrast with Theorem \ref{th2}, Proposition \ref{extraadded} relaxes the regularity condition on $H$, but imposes a divisibility condition on $N$ so that $n$ divides $N$.
\end{remark}

Proposition \ref{extraadded} routinely follows from Theorem \ref{th1} and Lemma \ref{simple} below. 

\begin{lemma}\label{simple}
If $\delta(H) \ge \frac{n-1}{2}$, then $\frac{a(k)}{k}$ is minimized if and only if $k = n$.
\end{lemma}

\begin{proof}
For $k \leq \frac{n}{2}$, fix $A \subset V(H)$ with $|A| = k$. As $\delta(H) \ge \frac{n-1}{2}$, $$\sum_{v \in A} d_H(v) \ge k \cdot \frac{n-1}{2}.$$
Clearly, each edge in the induced subgraph on $A$ contributes $2$ to $\sum_{v \in A} d_H(v)$. So,

\begin{equation}\label{eqn1}
a(A) \ge k \cdot \frac{n-1}{2} - \binom{k}{2} = k \cdot \frac{n - k}{2}.
\end{equation} 
Similarly, for $\frac{n}{2} < k < n$, fix $A \subset V(H)$ with $|A| = k$. As $\delta(H) \ge \frac{n-1}{2}$, for $v \in V(H) \setminus A$, 
$$d_A(v) \ge \frac{n-1}{2} - (n-k-1).$$ 
Hence, the number of edges between $A$ and $V(H) \setminus A$ is at least $(n-k)\left(\frac{n-1}{2} - (n-k-1)\right) = (n-k)\left(k - \frac{n-1}{2}\right)$.
Now because $\sum_{v \in A} d_H(v) \ge k \cdot \frac{n-1}{2}$,

\begin{align}\label{eqn2}
a(A) &\geq (n-k)\Big(k - \frac{n-1}{2}\Big) + \frac{1}{2} \bigg(k \cdot \frac{n-1}{2} - (n-k)\Big(k - \frac{n-1}{2}\Big)\bigg) \nonumber \\
&= \frac{1}{2}(n-k)\Big(k - \frac{n-1}{2}\Big) + \frac{1}{2} \cdot k \cdot \frac{n-1}{2} \nonumber \\
&> \frac{1}{2}(n-k)\Big(k - \frac{n-1}{2}\Big) + \frac{1}{2} \cdot (n-k) \cdot \frac{n-1}{2} \nonumber \\
&= k \cdot \frac{n - k}{2}.
\end{align}
Hence, by using \eqref{eqn1} and \eqref{eqn2}, we can conclude that,

\begin{equation}\label{eqn3}
a(k) \geq k \cdot \frac{n - k}{2}.
\end{equation}
Note that $a(n)$ is just the number of edges in $H$ and the number of edges in $H$ is at most $a(A) + \binom{n-|A|}{2}$ for any $A \subseteq V(H)$. Hence for any $k$,

\begin{equation}\label{eqn4}
 a(n) \leq a(k) + \binom{n-k}{2}.
 \end{equation}
Now from \eqref{eqn4} and \eqref{eqn3}, we can say that
$$a(n) < a(k) + (n-k) \cdot \frac{a(k)}{k}.$$ 
So, $$\frac{a(k)}{k} > \frac{a(n)}{n}.$$ 
\end{proof}

Now in order to prove Theorem \ref{th2}, we need a technical lemma. For $A \subseteq V(H)$, let $e(A)$ denote the number of edges in the induced subgraph on $A$, and let $e(k)$ denote the maximum value of $e(A)$ among all $k$-vertex subsets $A \subseteq V(H)$. 

\begin{lemma}\label{ext_le}
If $\delta(H) \ge \frac{n-1}{2}$, then for all positive integers $k$ and $l$ less than $n$, with $k+l \le n$, we have $e(k) + e(l) < e(k+l)$.
\end{lemma}

Before proving Lemma \ref{ext_le}, let us first state an equivalent form of it. Remember that $a(0) = 0$ from Definition \ref{defnA}.

\begin{corollary}\label{ext_co}
If $\delta(H) \ge \frac{n-1}{2}$, then for all positive integers $k$ and $l$ less than $n$, with $k+l \ge n$, we have $a(k) + a(l) > a(n) + a(r)$, where $k + l = n + r$.
\end{corollary}

\begin{proof}
Note that for any $A \subset V(H)$, $e(A) + a(V(H) \setminus A)$ is just the number of edges in $H$. So for each $k \le n$, $e(k) + a(n - k) = e(n) = a(n)$. Hence, we have the following: 
\begin{align*}
&e(k) + e(l) < e(k+l) \text{	for all	} k,l \le n \text{	and	} k+l \le n
\\&\Leftrightarrow \left(a(n) - e(k)\right) + \left(a(n) - e(l)\right) > 2a(n) - e(k+l) \text{	for all	} k,l \le n \text{	and	} k+l \le n
\\&\Leftrightarrow a(n-k) + a(n-l) > a(n) + a(n-k-l) \text{	for all	} k,l \le n \text{	and	} k+l \le n
\\&\Leftrightarrow a(k) + a(l) > a(n) + a(k+l-n) \text{	for all	} k,l \le n \text{	and	} k+l \ge n.
\end{align*}
\end{proof}

\begin{proof} [Proof of Lemma \ref{ext_le}]
Without loss of generality, assume that $k \le l$.

\medskip
\noindent \emph{\textbf{Case 1: $\boldsymbol{k \leq l \leq \frac{n}{2}}$.}}
Suppose $A \subseteq V(H)$ is an $l$-vertex subset such that $e(A) = e(l)$. Our strategy is to choose a $k$-vertex subset $B \subseteq V(H) \setminus A$ uniformly at random and show that the expected number of edges inside $A \cup B$ exceeds $e(k) + e(l)$. As $\delta(H) \ge \frac{n-1}{2}$, the number of edges in $H$ that are not inside the induced subgraph on $A$ is at least $\frac{n(n-1)}{4} - \binom{l}{2}$. Also, each vertex in $A$ can have at most $l-1$ neighbors in $A$, so the number of edges with one endpoint in $A$ and another endpoint in $V(H) \setminus A$ is at least $(\frac{n-1}{2} - (l-1))l$. Now if we choose a uniformly random subset $B \subseteq V(H) \setminus A$ of size $k$, the edges with one endpoint in $A$ and another endpoint in $V(H) \setminus A$ will be in $A \cup B$ with probability $\frac{k}{n-l}$, and the edges with both endpoints in $V(H) \setminus A$ will be in $A \cup B$ with probability $\frac{\binom{k}{2}}{\binom{n-l}{2}}$. Note that $\frac{k}{n-l} \ge \frac{\binom{k}{2}}{\binom{n-l}{2}}$. So, we have the following lower bound on the expected number of edges in $A \cup B$, which we denote by $\mathbb{E}[e(A \cup B)]$.

\begin{align*}
\mathbb{E}[e&(A \cup B)] \\
&\ge e(A) + \Big(\frac{n-1}{2} - (l-1)\Big)l \cdot \frac{k}{n-l} + \bigg(\frac{n(n-1)}{4} - \binom{l}{2} - \Big(\frac{n-1}{2} - (l-1)\Big)l\bigg) \frac{\binom{k}{2}}{\binom{n-l}{2}} \displaybreak[0] \\
&\ge e(A) + \frac{1}{2}(n-2l+1)(n-l) \cdot \frac{k^2}{(n-l)^2} + \bigg(\frac{n-1}{4}(n-2l) + \frac{l(l-1)}{2} \bigg) \frac{\binom{k}{2}}{\binom{n-l}{2}} \displaybreak[0] \\
&\ge e(A) + \frac{\binom{k}{2}}{\binom{n-l}{2}} \bigg( \frac{1}{2}(n-l)(n-2l+1) + \frac{n-1}{4}(n-2l) + \frac{l(l-1)}{2} \bigg) \displaybreak[0] \\
&= e(A) + \frac{\binom{k}{2}}{2(n-l)(n-l-1)} \left(2(n-l)(n-2l) + (2n - 2l) + n(n-2l) - (n-2l) + 2l^2 -2l\right) \displaybreak[0] \\
&> e(A) + \binom{k}{2} \frac{(n-2l)(3n-2l) + 2l^2}{2(n-l)^2}. \\
\end{align*}
If we parametrize $x = \frac{l}{n}$ then $\frac{(n-2l)(3n-2l) + 2l^2}{2(n-l)^2} = \frac{(1-2x)(3-2x) + 2x^2}{2(1-x)^2}$ which we define to be the function $f(x)$. Note that $0 \le x \le \frac{1}{2}$ because $l \le \frac{n}{2}$, and by routine calculus, it can be checked that $f(x) \ge 1$ for $0 \le x \le \frac{1}{2}$. So, $$\mathbb{E}[e(A \cup B)] > e(A) + \binom{k}{2} \geq e(l) + e(k).$$ Hence, there exists a $k$-vertex subset $B \subset V(H) \setminus A$ such that $e(A \cup B) > e(k) + e(l)$ where $|A \cup B| = k+l$, proving that $e(k+l) > e(k) + e(l)$.

\medskip
\noindent \emph{\textbf{Case 2: $\boldsymbol{k \leq \frac{n}{2} < l}$.}}
We proceed similarly to the first case. Suppose $A \subseteq V(H)$ is an $l$-vertex subset such that $e(A) = e(l)$. Clearly, $|V(H) \setminus A| = n-l$, and of course $e(V(H) \setminus A) \leq \binom{n-l}{2}$. The number of edges with one endpoint in $A$ and another endpoint in $V(H) \setminus A$ is at least $(n-l)\frac{n-1}{2} - 2 \cdot e(V(H) \setminus A)$. Now, the total number of edges in $H$ that are not inside the induced subgraph on $A$ is at least 
$$e(V(H) \setminus A) + (n-l)\frac{n-1}{2} - 2 \cdot e(V(H) \setminus A) \geq (n-l)\frac{n-1}{2} - \binom{n-l}{2} > \binom{n-l}{2}.$$
\\Now if we choose a uniformly random subset $B \subseteq V(H) \setminus A$ of size $k$, the edges with one endpoint in $A$ and another endpoint in $V(H) \setminus A$ will be in $A \cup B$ with probability $\frac{k}{n-l} \ge \frac{\binom{k}{2}}{\binom{n-l}{2}}$, and the edges with both endpoints in $V(H) \setminus A$ will be in $A \cup B$ with probability $\frac{\binom{k}{2}}{\binom{n-l}{2}}$. Hence, $\mathbb{E}[e(A \cup B)] > e(A) + \frac{\binom{k}{2}}{\binom{n-l}{2}} \binom{n-l}{2} = e(A) + \binom{k}{2}$. So, we are done.
\end{proof}

Now armed with the last lemma we are ready to prove a technical theorem with a more general condition than Theorem \ref{th2}, which will imply Theorem \ref{th2}.

\begin{theorem}\label{thatsit}
Suppose $H$ is a graph with $n$ vertices and minimum degree at least $\frac{n-1}{2}$, and $a(k) + a(l) > a(k+l)$ for all positive integers $k$ and $l$ with $k + l \leq n$. Then, $H$ is an ideal graph.
\end{theorem}

\begin{proof}
Fix positive integers $q,n$ and $N = qn + r$ with $0 \le r < n$. The assumption that $a(k) + a(l) > a(k + l)$ is the same as equation \eqref{conv_type_ineq_1} in the proof of Proposition \ref{pr}, and the conclusion of Corollary \ref{ext_co} is the same as equation \eqref{conv_type_ineq_2} in the proof of Proposition \ref{pr}. So, the integer program \eqref{IP} in Theorem \ref{th1} has the same unique minimum solution as in Proposition \ref{pr}, i.e., $x_n = q$, $x_r = 1$, and $x_k = 0$ for all other $k$. This implies that the extremal graph in Definition \ref{defn1} has the minimum number of edges among $N$-vertex $H$-covered graphs. Now the only remaining thing that we need to show is the uniqueness of the extremal graph. This proof is very similar to the uniqueness proof in Proposition \ref{pr}, which is routine to verify. 
\end{proof}

\begin{proof}[\textbf{Proof of Theorem \ref{th2}}]
Consider an arbitrary $d$-regular graph $H$. Observe that for any $A \subseteq V(H)$, we have $a(A) + e(A) = d|A|$. Hence, 
\begin{equation}\label{miracle}
a(k) + e(k) = dk \text{ for } k \le n.
\end{equation} 
If $d \ge \frac{n-1}{2}$, Lemma \ref{ext_le} implies that $e(k+l) > e(k) + e(l)$ for all $k,l > 0$ with $k + l \le n$. Now by using \eqref{miracle}, we can conclude that $a(k) + a(l) > a(k + l)$ for all $k,l > 0$ with $k + l \le n$. Hence, Theorem \ref{th2} is now implied by Theorem \ref{thatsit}.
\end{proof}

\subsection{Construction showing that Theorem \ref{th2} is tight}\label{tight}
In this subsection, we construct graphs to show the tightness of Theorem \ref{th2} in a strong sense. For every odd $n$, we will construct non-ideal regular graphs with $n$ vertices and degree as close as possible but less than $\frac{n-1}{2}$. We need a few notations for the ease of describing the graphs.
\begin{notation*}
\mbox{}
\begin{itemize}
\item Suppose $G_i$ is a graph for all $i \in [n]$ and $a_i$ is a non-negative integer for all $i \in [n]$. We denote the disjoint union of $a_i$ many $G_i$ for all $i \in [n]$ by $\oplus_{i=1}^n a_iG_i$. We simply write $G$ instead of $1G$.
\item Suppose $\S_i$ is a set of graphs for all $i \in [n]$. The set $\oplus_{i=1}^n\S_i$ denotes the set $\{\oplus_{i=1}^n H_i : H_i \in \S_i\}$.
\end{itemize}
\end{notation*}

Let $n = 2l+1$ and $d$ be an even number such that $d < \frac{n-1}{2}$. This is the highest we can consider because there are no graphs with an odd number of vertices and an odd degree of regularity. Consider a connected $d$-regular graph $H_1$ on $l$ vertices and a connected $d$-regular graph $H_2$ on $l+1$ vertices. Consider the $d$-regular graph $H = H_1 \oplus H_2$, which has $n$ vertices. 

\begin{lemma} \label{minimized}
$\frac{a(k)}{k}$ is minimized if and only if $k \in \{l,l+1,2l+1\}$.
\end{lemma}

\begin{proof}
As $H$ is $d$-regular, $a(k) \geq \frac{dk}{2}$. Now $a(V(H_1)) = \frac{dl}{2}$, $a(V(H_2)) = \frac{d(l+1)}{2}$ and $a(V(H)) = \frac{d(2l+1)}{2}$. Hence, for $k \in \{l, l+1, 2l+1\}$, $\frac{a(k)}{k}$ achieves its minimum. For $A \subset V(H)$ with $|A| \notin \{l, l+1, 2l+1\}$, there is at least one edge between $A$ and $V(H) \setminus A$, which implies that $a(A) > \frac{d|A|}{2}$. So, $\frac{a(k)}{k} > \frac{d}{2}$ for $k \notin \{l, l+1, 2l+1\}$.
\end{proof}

\begin{proposition}
$H$ is not ideal in any remainder class.
\end{proposition}

\begin{proof}
Let $N \ge n$ and $G$ be an $N$-vertex $H$-covered graph. Clearly, every vertex of $G$ has degree at least $d$ because every vertex of $G$ is in a copy of $H$, which is a $d$-regular graph. Hence, the number of edges of $G$ is at least $\frac{dN}{2}$. Now fix $r \in \{0,1,...,n-1\}$. Choose $a, b \in \mathbb{N}$ such that $a>1$ and $an + r = n + bl$. This choice is possible because $n = 2l+1$. Let $N = an + r = n + bl$. Consider the graph $G = H \oplus bH_1 = H_2 \oplus (b+1)H_1$, clearly every vertex of $G$ is in a copy of $H$. Notice that the number of edges in $G$ is $\frac{dN}{2}$, because the degrees of all vertices are $d$. Observe that $G$ is not isomorphic to any graph in $(a-1)H \oplus \L_{n+r}^H$, because there are at least $2$ disjoint copies of $H$ in any graph in $(a-1)H \oplus \L_{n+r}^H$. So, $H$ is not ideal. 
\end{proof}

\begin{remark}
Moreover, due to Lemma \ref{minimized} the number of edges in any graph in $(a-1)H \oplus \L_{n+r}^H$ is strictly greater than $\frac{dN}{2}$ when $r \notin \{0, l, l+1\}$. 
\end{remark}

\section{Random graph tools for proving Theorem \ref{th3}}\label{sec4}
This section will be full of facts and lemmas about $G_{n,p}$, where $0 < p < 1$ is a fixed constant. Throughout this paper we will only focus on constant probability $p$ in $G_{n,p}$ (in other words, $p$ will not depend on $n$). We will often denote $1-p$ by $q$. In this section and the next section, all logarithms are in base $e$. We will begin by mentioning some of the theorems and lemmas about the degree sequence of $G_{n,p}$ from \cite{BB} and \cite{FK}, which will be useful in the proof of Theorem \ref{th3}. Let the degree sequence of $G_{n,p}$ be $d_1 \le d_2 \le ... \le d_n$. The abbreviation w.h.p.\ stands for \textit{with high probability}, which means that the probability of the event under consideration tends to one as $n \rightarrow \infty$, as it was defined in the introduction. 

\begin{lemma}[Lemma 3.7 in \cite{FK}]\label{alan}
Let $\epsilon = 1/10$, and let $q = 1-p$. If 
$$d_{\pm} = (n-1)p - (1 \mp \epsilon)\sqrt{2(n-1)pq \log n}$$
then w.h.p.\
\begin{itemize}
\item The minimum degree $\delta(G_{n,p}) \ge d_{-}$.
\item There are $\Omega\left(n^{2 \epsilon (1 - \epsilon)}\right)$ vertices of degree at most $d_{+}$.
\item There are no $u \neq v$ such that $d_u, d_v \le d_{+}$ and $|d_u - d_v| \le 10$.
\end{itemize}
\end{lemma}

Next we mention a well-known fact about the difference between maximum and the minimum degree of $G_{n,p}$, which can be viewed as a corollary of Lemma \ref{alan}.

\begin{corollary}[Consequence of Lemma \ref{alan}]\label{af}
W.h.p.\ $G_{n,p}$ has the property that $d_n - d_1 < 4 \sqrt{(n-1)pq \log n}$, where $q = 1-p$.
\end{corollary}

\begin{theorem}[Theorem 12 in \cite{BB}]\label{degree}
Suppose $m \rightarrow \infty$ and $m = o(n)$. Put $q = 1-p$ and 
\begin{equation*}
K(m,n) = pn - \sqrt{2pqn \log (n/m)} + \big(\log \log (n/m) + \log 4 \pi \big) \sqrt{\frac{pqn}{8 \log (n/m)}}.
\end{equation*}
Then w.h.p.\ $G_{n,p}$ satisfies $|d_m - K(m,n)| = O\left(\sqrt{\frac{n}{\log (n/m)}}\right)$.
\end{theorem}

\begin{lemma}[Corollary 15 in \cite{BB}]\label{consecutive}
If $m = o(n^{1/4})/(\log n)^{1/4}$, then w.h.p.\ $G_{n,p}$ has the property $d_1 < d_2 < ... < d_m$.
\end{lemma}

\begin{lemma}[Lemma 18 in \cite{BB}]\label{sim}
For any $\epsilon > 0$, w.h.p.\ $G_{n,p}$ has the property that whenever $m > n^\epsilon$ every subgraph of order $m$ has $p \binom{m}{2} + o(m^2)$ edges.
\end{lemma}

\begin{corollary}[Consequence of Lemma \ref{sim}]\label{bip}
For any $\epsilon > 0$, w.h.p.\ $G_{n,p}$ has the property that whenever $m \geq l > n^\epsilon$, $A \in \binom{V(G_{n,p})}{m}$, $B \in \binom{V(G_{n,p})}{l}$, and $A \cap B = \emptyset$, then the number of edges with one endpoint in $A$ and another in $B$ is $p ml + o(m^2)$. 
\end{corollary} 

\begin{proof}[Proof Sketch]
The above corollary is an easy consequence of Lemma \ref{sim}. If $A$ and $B$ are disjoint vertex subsets of $G_{n,p}$, then the number of edges with one endpoint in $A$ and another in $B$ is exactly $e(A \cup B) - e(A) - e(B)$, where $e(V)$ denotes the number of edges in the induced subgraph on $V$. Now apply Lemma \ref{sim} on $A$, $B$, and $A\cup B$ to conclude Corollary \ref{bip}.
\end{proof}

\noindent \textbf{Chernoff Bound.}
We will refer to the following Chernoff-type bound (see \cite{H}) for the tail of the binomial distribution. Let $X$ be the sum of $n$ independent indicator random variables $X_i$ with $\mathbb{P}[X_i = 1] = p$, i.e. $X = \sum_{i=1}^n X_i$. Then for any $\epsilon > 0$,
\begin{equation}
\mathbb{P}[X \leq (1 - \epsilon)np] \leq e^{-\frac{\epsilon^2 n p}{2}}.\label{chern}
\end{equation} 

\begin{lemma}\label{chernoff}
For any $\epsilon > 0$, w.h.p.\ $G_{n,p}$ has the property that for all $A \subset V(G_{n,p})$ with $0 < |A| < n$, there exists $v \in V(G_{n,p}) \setminus A$ such that the number of neighbors of $v$ in $A$ is more than $(1 - \epsilon) p |A|$. 
\end{lemma}
\begin{proof}
It is enough to prove that w.h.p.\ for all $A \in V(G_{n,p})$ with $1 \leq |A| \leq \frac{n}{2}$, the number of edges between $A$ and $V(G_{n,p}) \setminus A$ is more than $(1 - \epsilon) p |A| (n-|A|)$. Fix $A \subset V(G_{n,p})$ with $1 \leq |A| \leq \frac{n}{2}$, and let $m = |A|$. Now by the Chernoff Bound \eqref{chern}, the probability that there are less than $(1 - \epsilon) p m (n-m)$ edges between $A$ and $V(G_{n,p}) \setminus A$ is less than $e^{-\frac{m (n-m) \epsilon^2 p}{2}}$. So, the probability that there exists $A$ such that $1 \leq |A| \leq \frac{n}{2}$ and the number of edges between $A$ and $V(G_{n,p}) \setminus A$ is less than $(1 - \epsilon) p |A| (n-|A|)$ is less than 
$$\sum_{m=1}^{\floor{\frac{n}{2}}} \binom{n}{m} e^{-\frac{m(n-m)\epsilon^2p}{2}} \leq \sum_{m=1}^{\floor{\frac{n}{2}}} \left(\frac{ne}{m} \cdot e^{-\frac{n\epsilon^2p}{4}}\right)^m = o(1).$$
\end{proof}

The next two lemmas are not tight, but give us what will suffice to prove our results.

\begin{lemma}\label{gap}
Let the degree sequence of $G_{n,p}$ be $d_1 \le d_2 \le ... \le d_n$. Then there exists $\epsilon>0$ such that w.h.p.\ for all $2 \le k \le \epsilon \sqrt{n \log n}$, $\sum_{i = 1}^k (d_i - d_1) > \binom{k}{2}$.
\end{lemma}
\begin{proof}
Let $m = n^{\frac{1}{5}}$. By Lemma \ref{consecutive}, w.h.p.\ $d_1 < d_2 < d_3 < ... < d_{3m}$. Note that we should have used $\floor{3m}$ or $\ceil{3m}$ instead of $3m$ but it is true in both cases, hence for asymptotic results we omit rounding when it is not essential. From Lemma \ref{alan}, it can be seen that w.h.p.\ $d_2 - d_1 \ge 2$. Note that $\binom{k}{2} = \sum_{i=1}^{k-1} i$, hence we have proven the statement of Lemma \ref{gap} for $k \le 3m$. 
Now using Theorem \ref{degree}, we have the following w.h.p., 
\begin{align*}
d_m - d_1 &\ge d_m - d_{\sqrt{m}} \\
&\ge pn - \sqrt{2pqn \log (n/n^{1/5})} - pn + \sqrt{2pqn \log (n/n^{1/10})} - O\left(\sqrt{n}\right) \\
&\ge \frac{\sqrt{9} - \sqrt{8}}{\sqrt{10}} \sqrt{2pqn \log n} - O\left(\sqrt{n}\right).
\end{align*} 
Fix a constant $0< \epsilon < \frac{\sqrt{9} - \sqrt{8}}{\sqrt{10}} \cdot \sqrt{2pq}$. Now for $3m < k < \epsilon \sqrt{n \log n}$, at least half of the $i \in [k]$ have $d_i \ge d_m$, so $\sum_{i = 1}^k (d_i - d_1) > \frac{k}{2}(d_m - d_1) \ge \frac{k}{2} \epsilon \sqrt{n \log n} > \binom{k}{2}$.
\end{proof}

\begin{lemma}\label{gap1}
Let the degree sequence of $G_{n,p}$ be $d_1 \le d_2 \le ... \le d_n$. Then w.h.p.\ for all $n^{\frac{1}{2}} < k < \frac{1}{2} n^{\frac{1}{2}} \log n$, $\sum_{i = 1}^k (d_{2k} - d_i) = o(k^2)$.
\end{lemma}
\begin{proof}
\begin{align}
\sum_{i = 1}^k d_{2k} - d_i 
&\le \sum_{i = 1}^{\frac{\sqrt{n}}{\log n}} (d_{\sqrt{n} \log n} - d_i) + \sum_{i = \frac{\sqrt{n}}{\log n}}^k (d_{\sqrt{n} \log n} - d_i)\nonumber
\\&\le \frac{\sqrt{n}}{\log n} \cdot 4 \sqrt{pqn \log n} + k (d_{\sqrt{n} \log n} - d_{\frac{\sqrt{n}}{\log n}})\label{c}
\\&\le o(n) + k\Bigg(\sqrt{2pqn \log \bigg(\frac{n \log n}{\sqrt{n}}\bigg)} - \sqrt{2pqn \log \bigg(\frac{n}{\sqrt{n} \log n}\bigg)} + o(\sqrt{n})\Bigg)\label{deg}
\\& = o(k^2).\nonumber
\end{align}

In the above calculation, in step \eqref{c}, we have used the concentration of the degree sequence of $G_{n,p}$ (Corollary \ref{af}), and in step \eqref{deg}, we have used Theorem \ref{degree}.
\end{proof}

Now we will switch to a natural problem of independent interest that turns out to be very useful to our problem. We need to use and build upon a couple of definitions from Section 3.1 to describe the problem.

\begin{definition}\label{dense}
\mbox{}
\begin{itemize}
\item For a graph $G$ and a subset $A$ of its vertex set, let $e(A)$ denote the number of edges in the induced subgraph on $A$.
\item For a graph $G$, $e_G(k)$ is defined to be the maximum value of $e(A)$ among all $k$-vertex subsets $A \subseteq V(G)$. When $G$ is clear from the context, we will write $e(k)$ instead of $e_G(k)$. In all random graph situations in this paper, $p$ will be a fixed constant as $n$ grows. With $p$ being understood as given, we will write $e_n(k)$ to denote the random variable $e_{G_{n,p}}(k)$. 
\end{itemize}
\end{definition}

Our aim is to obtain a concentration-type result for $e_n(k)$, where the following inequality will become handy. This version of Talagrand's Inequality can be found in Section 21.8 of \cite{FK}.
\begin{theorem}[Talagrand's Inequality]\label{talagrand}
Let $\Omega = \prod_{i=1}^n \Omega_i$, where each $\Omega_i$ is a probability space and $\Omega$ has the product measure. Let $\A \subseteq \Omega$ and let $x = (x_1,x_2,...,x_n) \in \Omega$.
\\For $\alpha = (\alpha_1,\alpha_2,...,\alpha_n) \in \mathbb{R}^n$ we let 
$$d_\alpha(\A,x) = \inf_{y \in \A} \sum_{i : y_i \neq x_i} \alpha_i.$$
Then we define $$\rho(\A,x) = \sup_{|\alpha| = 1} d_\alpha(\A,x),$$
\\where $|\alpha|$ denotes the Euclidean norm $\sqrt{{\alpha_1}^2 + {\alpha_2}^2 + ... + {\alpha_n}^2}$.
\\We then define, for $t \ge 0$, $$\A_t = \{x \in \Omega : \rho(\A,x) \le t\}.$$
\\Then $$\mathbb{P}[\A](1 - \mathbb{P}[\A_t]) \le e^{-\frac{t^2}{4}}.$$
\end{theorem}

\begin{lemma}\label{concentration}
For any function $f(n)$ such that $f(n)$ tends to infinity as $n$ tends to infinity, w.h.p.\ the difference between $e_n(k)$ and $e_{\frac{n}{2}}(k)$ is at most $k^\frac{7}{4}$ for all $f(n) \leq k \leq \frac{n}{10}$.
\end{lemma}

\begin{proof} 
We first mention a brief outline of the proof. We will use Talagrand's Inequality \ref{talagrand} to show that $e_{\frac{n}{2}}(k)$ is concentrated in an interval of length $k^{\frac{7}{4}}$ for all $c \log n \leq k \leq \frac{n}{2}$  with (very) high probability. Then we will show that there are $V_1,V_2,...,V_l \in \binom{V(G_{n,p})}{n/2}$ (where $l$ is not very big) such that any $k$-subset of $V(G_{n,p})$ will be totally inside at least one $V_i$. Note that the induced graph on $V_i$ is just distributed as $G_{\frac{n}{2},p}$, hence $\mathbb{P}[e_n(k) > x] \leq l \cdot \mathbb{P}[e_{n/2}(k) > x]$ by the union bound. This is the strategy to show a small gap between $e_n(k)$ and $e_{n/2}(k)$ for fixed $k$, and of course to prove the same for all $f(n) \le k \le n$, we use the union bound.

\medskip
\noindent \textbf{Concentration of $\boldsymbol{e_{\frac{n}{2}}(k)}$.}
Let $m = \binom{n/2}{2}$. The $G_{\frac{n}{2},p}$ model is the same as $\Omega = \prod _{i = 1}^{m} \Omega_i$, where the $\Omega_i$'s are independent Bernoulli Distributions with success probability $p$. Let 
\begin{equation}\label{qk}
x_k = \min\{x \in \mathbb{N} : \mathbb{P}[e_{n/2}(k) \leq x] \geq q_k\},  \text{     where     } q_k = e^{-\frac{k^{3/2}}{8}}.
\end{equation}
Let $\A = \{G \in \Omega : e_G(k) \leq x_k\}$. Clearly 
\begin{equation}\label{tail}
\mathbb{P}[\A] \geq q_k.
\end{equation} 
\\For $G = (G_1, G_2, \cdots, G_m) \in \Omega$, let $$\rho (\A, G) = \sup_{|\alpha| = 1} \inf_{H \in \A} \sum_{i : G_i \neq H_i} {\alpha}_i.$$ 
\\For $t > 0$, let $$\A_t = \{G \in \Omega : \rho (\A, G) \leq t\}.$$ 
\\If $G \in \Omega$ such that $e_G(k) \geq x_k + d_k$, then there exists a $k$-vertex subset $A \subseteq V(G)$ such that $e(A) \geq x_k + d_k$. So, by using $\alpha_i$'s such that those corresponding to pairs of vertices in $A$ are $\alpha_i = \frac{1}{\sqrt{\binom{k}{2}}}$ and all other $\alpha_i = 0$, it is easy to see that $G \notin \A_{\frac{d_k}{k}}$.
\\Using Talagrand's Inequality (Theorem \ref{talagrand}), $\mathbb{P}[\A](1 - \mathbb{P}[\A_t]) \leq e^{-\frac{t^2}{4}}$, and using \eqref{tail}, we can conclude that $q_k \mathbb{P}\left[e_{n/2}(k) \geq x_k + d_k\right] \leq e^{-\frac{{d_k}^2}{4k^2}}$. Now choosing $d_k = k^{\frac{7}{4}}$ and using \eqref{qk}, we have $\mathbb{P}\left[e_{n/2}(k) \geq x_k + k^{\frac{7}{4}}\right] \leq e^{-\frac{k^{3/2}}{8}}$. Hence, 
\begin{equation}\label{half}
\mathbb{P}\left[e_{n/2}(k) \notin \left[x_k, x_k + k^{\frac{7}{4}}\right]\right] \leq 2 e^{-\frac{k^{3/2}}{8}}.
\end{equation}
Since $f(n) \rightarrow \infty$ as $n \rightarrow \infty$, the probability that there exists $f(n) \le k \le \frac{n}{2}$ such that $e_{n/2}(k) \notin [x_k, x_k + k^{\frac{7}{4}}]$ is at most 
$\sum_{k = f(n)}^{\infty} 2 e^{-\frac{k^{3/2}}{8}} = o(1)$.

\medskip
\noindent \textbf{Covering of $\boldsymbol{V}$ ($\boldsymbol{|V| = n}$) by subsets of size $\boldsymbol{\frac{n}{2}}$.}
We will show that there exists $l < 2^{3k}$ and $V_1,V_2,...,V_l \in \binom{V}{\frac{n}{2}}$ such that for all $A \in \binom{V}{k}$ there exists $i \in [l]$ with $A \subset V_i$. First partition $V$ into $3k$ almost-equal sets $U_1,U_2,...,U_{3k}$ such that $|U_i|$ equals to either $\ceil{\frac{n}{3k}}$ or $\floor{\frac{n}{3k}}$ for all $i \in [3k]$. Now for all $S \in \binom{[3k]}{k}$, choose $V_S \in \binom{V}{\frac{n}{2}}$ such that $V_S \supseteq \bigcup_{i \in S} U_i$, which is clearly possible because $\sum_{i \in S} |U_i| < k(\frac{n}{3k} + 1) < \frac{n}{2}$ due to the assumption that $n \le \frac{n}{10}$. Note that $|S| = \binom{3k}{k} < 2^{3k}$. Now we show that for $A \in \binom{V}{k}$ there exists $S \in \binom{[3k]}{k}$ such that $A \subset V_S$. For a fixed $A \in \binom{V}{k}$, consider $S' = \{i : U_i \cap A \neq \emptyset\}$; clearly $|S'| \leq k$ and so any $S \in \binom{[3k]}{k}$ with $S \supseteq S'$ will obviously give $A \subset V_S$. 

\medskip
\noindent \textbf{\textbf{Completing proof of Lemma \ref{concentration}.}}
Note that the distribution of an induced subgraph of $G_{n,p}$ on vertex set $V \subseteq V(G_{n,p})$ is $G_{|V|,p}$. For a graph $G$ with $n$ vertices using the above covering of vertex set $V(G)$, we know that there exist at most $2^{3k}$ induced subgraphs on $\frac{n}{2}$ vertices such that if all of those subgraphs do not have any subgraph of size $k$ with at least $M$ edges then $G$ also does not have a subgraph of size $k$ with at least $M$ edges. Hence, by using \eqref{half},
\begin{align*}
\mathbb{P}\left[e_n(k) > x_k + k^{\frac{7}{4}}\right] 
&\le 2^{3k} \mathbb{P}\left[e_{n/2}(k) > x_k + k^{\frac{7}{4}}\right]\\
&\le 2^{3k} 2 e^{-\frac{k^{3/2}}{8}}\\
&\le e^{-\frac{k^{3/2}}{9}}.
\end{align*}
It is also trivial to see that $$\mathbb{P}[e_n(k) < x_k] \le \mathbb{P}[e_{n/2}(k) < x_k] \le e^{-\frac{k^{3/2}}{9}}.$$
Since $f(n) \rightarrow \infty$ as $n \rightarrow \infty$, by a similar summation over $k$ as was done for $e_{n/2}(k)$, one can see that w.h.p.\ $e_n(k)$ also takes a value in the interval $\left[x_k, x_k + k^{\frac{7}{4}}\right]$ for all $f(n) \le k \le \frac{n}{10}$. Hence, we are done.
\end{proof}

\section{Proof of Theorem \ref{th3}}\label{theorem3}\label{sec5}
Let $0 < p < 1$ be a fixed constant. Fix a graph $H$ sampled from the distribution of random graphs $G_{n,p}$. Throughout this section, let $d_1 \le d_2 \le ... \le d_n$ be the degree sequence of $H$, and we will use $a(A)$ and $a(k)$ from Definition \ref{defnA}. Let $\beta'$ denote the minimum $k \in \{2, 3, \cdots, n\}$ for which $a(k) \le k a(1)$. Note that this inequality is trivially true for $k = 1$, but w.h.p.\ false for $k = 2$ since w.h.p.\ $d_2 - d_1 \ge 2$ by Lemma \ref{alan}. Lemma \ref{beta} will show that w.h.p.\ $\beta'$ exists. Let $\gamma = \beta' - 1$ and let us define $\beta$ as follows. 
\begin{itemize}
\item If $a(\beta') = \beta' a(1)$, then define $\beta = \beta' + 1$.
\item Otherwise define $\beta = \beta'$.
\end{itemize}
Our aim is to show that this choice of $\gamma$ and $\beta$ works in Theorem \ref{th3}. From the definition of $\gamma$ we have the following for all $2 \le k \le \gamma$: 
\begin{equation}\label{single}
a(k) > k a(1).
\end{equation}

\begin{lemma}\label{beta}
W.h.p.\ $\gamma = \Theta(\sqrt{n \log n})$.
\end{lemma}

\begin{proof}
By using Lemma \ref{gap}, there exists $\epsilon > 0$ such that w.h.p.\ for all $2 \le k \le \epsilon \sqrt{n \log n}$, 
\begin{equation}\label{ll}
a(k) \geq \sum_{i = 1}^k d_i - \binom{k}{2} > kd_1 = ka(1).
\end{equation}
Now there exists a constant $L$ such that w.h.p.\ for $k = L \sqrt{n \log n}$,
\begin{align}
a(k) &< \sum_{i = 1}^k d_i - \frac{p}{2}\binom{k}{2}\label{ll1}
\\&< kd_1 + k \Theta(\sqrt{n \log n}) - \frac{p}{2}\binom{k}{2}\label{ll2}
\\&< kd_1\label{ll3}
\\&= ka(1),\label{ll4}
\end{align}
where \eqref{ll1} is true because w.h.p.\ if $A \in \binom{V(H)}{k}$ consists of the $k$ vertices with smallest degrees, then $a(A) < \sum_{i = 1}^k d_i - \frac{p}{2}\binom{k}{2}$, where the $\frac{p}{2}\binom{k}{2}$ part comes from Lemma \ref{sim}. Step \eqref{ll2} comes from the fact that w.h.p.\ $d_n - d_1 \le \Theta(\sqrt{n \log n})$, which is true because of Corollary \ref{af}. One can choose a constant $L$ satisfying \eqref{ll3}. Now combining \eqref{ll} and \eqref{ll4}, we conclude that w.h.p.\ $\gamma = \Theta(\sqrt{n \log n})$.
\end{proof}

\begin{lemma}\label{helper}
W.h.p.\ for all positive integers $k$ and $l$ such that $n^{\frac{1}{2}} < k < l \le n$, it holds that $\frac{a(k)}{k} > \frac{a(l)}{l}$.
\end{lemma}

\begin{proof}
It is enough to show that w.h.p.\ for all $k > n^{\frac{1}{2}}$, it holds that $\frac{a(k)}{k} > \frac{a(k+1)}{k+1}$. By Lemma \ref{sim}, w.h.p.\ for an arbitrary $k$-vertex subset $A \subset V(H)$, we have $e(A) < \frac{2}{3} p k^2$. So, $$a(A) = \sum_{v \in A} d_v - e(A) > \sum_{i=1}^k d_i - \frac{2}{3} p k^2.$$ 
Hence,
\begin{equation}\label{fst}
a(k) > \sum_{i=1}^k d_i - \frac{2}{3} p k^2.
\end{equation}

Let us first prove that w.h.p.\ $\frac{a(k)}{k} > \frac{a(k+1)}{k+1}$ for all $n^{\frac{1}{2}} < k < \frac{1}{2} n^{\frac{1}{2}} \log n$. Let $A \in \binom{V(H)}{k}$ such that $a(k) = a(A)$. Let $B$ be the set of the $k$ smallest degree (in $H$) vertices outside $A$. Now by using Corollary \ref{bip}, there are more than $\frac{5}{6} p k^2$ many edges between $A$ and $B$. Hence, there is $v \in B$ such that $v$ has more than $\frac{5}{6} p k$ neighbors in $A$. So, $$a(A \cup \{v\}) = a(A) + d_v - N_A(v) < a(k) + d_{2k} - \frac{5}{6} p k.$$ 
Hence, 
\begin{equation}\label{snd}
a(k+1) < a(k) + d_{2k} - \frac{5}{6} p k.
\end{equation}
Now adding inequality \eqref{fst} to inequality \eqref{snd} multiplied by $k$, we get $$ka(k+1) < (k+1)a(k) + \sum_{i = 1}^k (d_{2k} - d_i) - \frac{1}{6} p k^2.$$ So, in order to prove that w.h.p.\ $\frac{a(k)}{k} > \frac{a(k+1)}{k+1}$, it is enough to show that $\sum_{i = 1}^k (d_{2k} - d_i) < \frac{1}{6} p k^2$, which is w.h.p.\ true because of Lemma \ref{gap1}.

For the case $k \ge \frac{1}{2} n^{\frac{1}{2}} \log n$, instead of needing to carefully pick which vertex to add to $A$, we can pick an arbitrary vertex. Let $A \in \binom{V(H)}{k}$ such that $a(k) = a(A)$. Now use Lemma \ref{chernoff} to conclude that w.h.p.\ there is $v \in V(H) \setminus A$ such that $N_A(v) > \frac{5}{6} p k$. So, similar to the other case but with $d_n$ instead of $d_{2k}$ because we pick an arbitrary vertex, it is sufficient to show that $\sum_{i = 1}^k (d_n - d_i) < \frac{1}{6} p k^2$, which is w.h.p.\ true because of the fact that w.h.p.\ $d_n - d_1 = \Theta(\sqrt{n \log n})$ from Corollary \ref{af}.
\end{proof}

From Lemma \ref{beta} and Lemma \ref{helper}, we can say that w.h.p.\ for $k \ge \beta$,
\begin{equation}\label{mid}
a(k) < ka(1).
\end{equation}

\begin{lemma}\label{double}
W.h.p.\ for all positive integers $k$ and $l$ with $\beta \le k+l \le n$, we have $a(k) + a(l) > a(k+l)$.
\end{lemma}

\begin{proof}
First we claim that for any positive integers $k'$ and $l'$ with $k' < l'$ and $\beta \le l'$, it holds that $\frac{a(k')}{k'} > \frac{a(l')}{l'}$. This clearly holds when $k' > n^{\frac{1}{2}}$ by Lemma \ref{helper}. Otherwise $k' \le \gamma$ because $\gamma = \Theta(\sqrt{n \log n})$ by Lemma \ref{beta}, and in this case by \eqref{single} and \eqref{mid}, we have $\frac{a(k')}{k'} \ge a(1) > \frac{a(l')}{l'}$. Now by using this claim, for any positive integers $k$ and $l$ with $\beta \le k+l \le n$, $$a(k) + a(l) > \frac{k}{k+l} \cdot a(k+l) + \frac{l}{k+l} \cdot a(k+l) = a(k+l).$$
\end{proof}

\begin{lemma}\label{prob}
W.h.p.\ for all positive integers $k$ and $l$ less than $n$ with $k+l \ge n$, we have $a(k) + a(l) > a(n) + a(r)$, where $k + l = n + r$.
\end{lemma}

\begin{remark}
If $p > \frac{1}{2}$, then w.h.p.\ the minimum degree of $G_{n,p}$ is at least $\frac{n}{2}$, and in this case Corollary \ref{ext_co} already implies Lemma \ref{prob}. Regardless, the proof below works for all $0 < p < 1$.
\end{remark}

\begin{proof}
As we defined in Definition \ref{dense}, $e(k) = \max_{A \in \binom{V(H)}{k}} e(A)$, where $e(A)$ denotes the number of edges in the induced subgraph on $A$. Just as in Corollary \ref{ext_co}, to prove Lemma \ref{prob}, it is enough to show that $e(k) + e(l) < e(k+l)$ whenever $k + l \le n$. Without loss of generality, let us assume that $k \ge l$.

\noindent \emph{\textbf{Case 1: $\boldsymbol{k \ge \frac{l}{p}}$.}}
Let $A \in \binom{V(H)}{k}$ such that $e(k) = e(A)$. Let $A_0 = A$, and define $A_{i+1}$ recursively by adding some $v \in V(H) \setminus A_i$ to $A_i$ such that $N_{A_i}(v) > \frac{p}{2} |A_i|$, which we know exists by Lemma \ref{chernoff}. Now $$e(A_l) > e(k) + \frac{p}{2} k l > e(k) + \binom{l}{2} \ge e(k) + e(l).$$ Note that $|A_l| = k + l$, and hence $e(k) + e(l) < e(k + l)$.

\noindent \emph{\textbf{Case 2: $\boldsymbol{k < \frac{l}{p}}$ and $\boldsymbol{k \geq n^{\frac{1}{2}}}$.}}
We use Lemma \ref{sim} to conclude that w.h.p.\ for all $l \le k < \frac{l}{p}$ with $k \ge n^{\frac{1}{2}}$, we have the following. 
$$e(k) < p\binom{k}{2} + o(k^2), \text{	} e(l) < p\binom{k}{2} + o(k^2), \text{	} e(k+l) > p\binom{k+l}{2} - o\left((k+l)^2\right).$$
Hence we have that 
\begin{align*}
e(k+l) &> p\binom{k+l}{2} - o\left((k+l)^2\right) \\
&= p\binom{k}{2} + p\binom{l}{2} + pkl - o\left(k^2\right) \\
&> p\binom{k}{2} + p\binom{l}{2} + p^2k^2 - o\left(k^2\right) \\
&> e(k) + e(l).
\end{align*}

\noindent \emph{\textbf{Case 3: $\boldsymbol{k < \frac{l}{p}}$ and $\boldsymbol{k < n^{\frac{1}{2}}}$.}}
Let $b = \frac{1}{p}$. It is well-known that the largest clique in $G_{n,p}$ has order $\approx 2 \log_{b} n$ (see, e.g., \cite{M} or Theorem 7.3 in \cite{FK}). So, it is enough to handle this case only when $k + l \ge \log_{b} n$.
Fix $k,l \in \mathbb{N}$ such that $k < \frac{l}{p}$, $k < n^{\frac{1}{2}}$ and $k + l \ge \log_{b} n$. Partition the vertex set $V(H)$ into two almost equal sets $V_1,V_2$ such that $||V_1| - |V_2|| \le 1$. Now the strategy is to first reveal the edges (edges are there with probability $p$ independently) in the induced subgraphs on $V_1$ and $V_2$, then find the densest $k$-vertex subset $A_1 \subseteq V_1$ and $l$-vertex subset $A_2 \subseteq V_2$, and finally after revealing all the edges between $A_1$ and $A_2$ we argue that the number of edges between $A_1$ and $A_2$ cannot be too small. Let $H_i$ denote the induced subgraph on $V_i$ for $i = 1,2$. Clearly, the distribution of the induced subgraph on $V \subset V(G_{n,p})$ is same as $G_{|V|,p}$.
From Lemma \ref{concentration}, we know that for $j = k,l$, 
\begin{equation}\label{eqt1}
e_H(j) - e_{H_i}(j) \leq j^{\frac{7}{4}} \text{ for } i = 1,2.
\end{equation} 
Let $A_1 \in \binom{V_1}{k}$ such that $e_{H_1}(k) = e(A_1)$, and let $A_2 \in \binom{V_2}{l}$ such that $e_{H_2}(l) = e(A_2)$. By the Chernoff Bound \eqref{chern}, 
\begin{equation}\label{eqt2}
\mathbb{P}\left[e(A_1, A_2) \leq \frac{1}{2} kl p \right] \le e^{-\frac{klp}{8}} \le e^{-\frac{p^2k^2}{8}} \le e^{-\frac{p^2(\log_b n)^2}{32}},
\end{equation}
where $e(A_1, A_2)$ denotes the number of edges between $A_1$ and $A_2$. If $e(A_1, A_2) > \frac{1}{2} kl p > \frac{p^2}{2} k^2 $, then 
\begin{align}
e_H(k + l) &\geq e_H(A_1 \cup A_2) \nonumber
\\&= e(A_1) + e(A_2) + e(A_1,A_2) \nonumber
\\&> e_{H_1}(k) + e_{H_2}(l) + \frac{p^2}{2} k^2 \nonumber
\\&\geq e_H(k) - k^{\frac{7}{4}} + e_H(l) - l^{\frac{7}{4}} + \frac{p^2}{2} k^2 \label{smallmatter}
\\&> e_H(k) + e_H(l), \nonumber
\end{align}
where step \eqref{smallmatter} uses \eqref{eqt1}.
Now by using \eqref{eqt2}, the probability that there exist $k,l$ in this case with $e(k) + e(l) \geq e(k+l)$ is at most $n e^{-\frac{p^2(\log_b n)^2}{32}} = o(1)$. So, we are done.
\end{proof}

\begin{proof}[\textbf{Completing proof of Theorem \ref{th3}.}]
By using \eqref{single}, \eqref{mid}, Lemma \ref{double}, and Lemma \ref{prob}, the solutions of the integer program \eqref{IP} in Theorem \ref{th1} for the graph $H$ (which was sampled from $G_{n,p}$) are given by
\begin{itemize}
\item If $N = qn + r$ with integers $q > 0$ and $0 \le r \le \gamma$, then 
\begin{equation}\label{couple1}
x_n = q, x_1 = r, \text{ and all other } x_k = 0.
\end{equation}
\item If $N = qn + r$ with integers $q > 0$ and $\beta \le r \le n-1$, then 
\begin{equation}\label{couple2}
x_n = q, x_r = 1, \text{ and all other } x_k = 0.
\end{equation}
\item If $a(\gamma + 1) = (\gamma + 1)a(1)$, and $N = qn + r$ with integers $q > 0$ and $r = \gamma + 1$, then both \eqref{couple1} and \eqref{couple2} are solutions, as it was remarked after the statement of Theorem \ref{th3} in Section 1.
\end{itemize}
Now by using Theorem \ref{th1} and analyzing Algorithm \ref{algo} for $N$-vertex $H$-covered graphs like we did for Proposition \ref{pr}, we can show the uniqueness of the extremal graphs, finishing the proof of Theorem \ref{th3}.
\end{proof}

\section{Concluding remarks} \label{sec6}
\begin{itemize}
\item It will be interesting to consider the problem of maximizing the number of independent sets of order $t > 2$ in an $n$-vertex $H$-covered graph. We have some initial observations that the structure of the optimal graph might be drastically different for $t > 2$. It seems to be a fairly different problem than the case of $t=2$, which is one of the reasons we did not include any of the analysis in this paper.  
\item It is not hard to see that trees and cycles are ideal by a straightforward application of Theorem \ref{th1}. It will be interesting to classify all graphs that are ideal. In Theorem \ref{th2}, we have seen that $n$-vertex $d$-regular graphs with $d \ge \frac{n-1}{2}$ are ideal. It might be worth searching for other nice classes of ideal graphs.  
\item On the generalized problem of minimizing the number of copies of a graph $T$ in an $N$-vertex $H$-covered graph, it will be interesting to see what kind of arguments can be made when $T$ is not a complete graph. For example, does Corollary \ref{sat} still hold? The integer program like \eqref{IP} in Theorem \ref{th1} will still give a lower bound, but may not be the correct answer because there is the possibility of extra copies of $T$ in $\L_{T,N}^H$, i.e., in Definition \ref{intersect} the graph $\L_{T,N}^H$ may have some extra copies of $T$ having vertices both in $V_1$ and $V_2$, which did not happen when $T$ was complete. 
\item In the context of Theorem \ref{th2}, it will be nice to investigate if the random $d$-regular graph $G_{n,d}$ for $d < \frac{n-1}{2}$ is ideal with high probability. It is not hard to notice that $G_{n,2}$ is not ideal w.h.p.\ by using the fact that $G_{n,2}$ is just the union of disjoint cycles of at least two different orders. It is worth mentioning that there exist connected $d$-regular graphs that are not ideal.
\item An interesting task will be to consider Theorem \ref{th3} when $p$ is a function of $n$ in $G_{n,p}$. Although our proof extends for some range of $p$ tending to zero, we leave this open for further investigation.
\end{itemize}

\section{Acknowledgement}
We thank Stijn Cambie and Ross Kang for pointing out a few minor problems in the previous version of this paper. We are grateful to the anonymous referees for helping us improve the exposition of this paper.

\end{document}